\documentclass[a4paper,10pt]{article}

\usepackage{amsmath,amssymb,amsthm,mathrsfs}
\usepackage{amsfonts}

\usepackage{makeidx}
\makeindex

\usepackage{url}

\theoremstyle{plain}
\newtheorem{thm}{Theorem}[section]
\newtheorem{lem}[thm]{Lemma}
\newtheorem{prop}[thm]{Proposition}
\newtheorem{cor}[thm]{Corollary}
\newtheorem*{tpeq}{Topological Equivalence Classification}
\newtheorem*{mt}{Main Theorem}

\theoremstyle{definition}
\newtheorem{defn}[thm]{Definition}

\theoremstyle{remark}
\newtheorem{rem}[thm]{Remark}

\usepackage[all]{xy}
\usepackage{graphicx}
\usepackage{subfigure}

\newcommand{\bb}{\mathbb}
\newcommand{\ca}{\mathcal}

\newcommand{\ti}{\tilde}
\newcommand{\tti}{\widetilde}

\title{Topological Conjugacy of Non-hyperbolic Linear Flows}
\author{Tongmu He\thanks{Department of Mathematical Sciences, Tsinghua University, Beijing, China}\thanks{Email address: hetm15@mails.tsinghua.edu.cn}}

\usepackage{hyperref}
\hypersetup{colorlinks,
citecolor=black,
filecolor=black,
linkcolor=blue,
urlcolor=black,
pdftitle={Topological Conjugacy of Non-hyperbolic Linear Flows},
pdfauthor={Tongmu He},
pdftex}

\begin{document}
\maketitle

\begin{abstract}
The topological equivalence classification for linear flows on $\bb{R}^n$ had been completely solved by Kuiper and independently Ladis in 1973. However, Ladis' proof was published in a Russian journal which isn't easily available, Kuiper's proof is more topological and a little bit subtle. Aiming at topological conjugacy classification, mainly based on the ideas of Kuiper, we introduce other techniques and try to present an elementary and self-contained proof just using linear algebra and elementary topology.
\end{abstract}

\section{Introduction}

A differential equation always can be written in autonomous form $\dot{x} = v(x)$ where $x\in \bb{R}^n$. And its solution orbits form some geometry near a point. If $v(x_0)\neq 0$, basic knowledge in ordinary equations tells us we can rectify the vector field in a neighborhood of $x_0$ and then we know the geometry near such point (non-singular) is trivial. So we want to deal with the case $v(x_0)=0$ (singular point).

Note that linear field approximates a general field, which lead us to the case $\dot{x} = Ax$ ($A$ is a linear operator in $\bb{R}^n$). Its solutions are $x(t)=e^{tA}x_0$ where $x_0$ are initial points. We want to classify the topology of the solutions i.e. the flows $e^{tA}$ near $x_0=0$, therefore we introduce two types of equivalence relations on the topology of flows: topologically equivalent and topologically conjugate (see section \ref{preli}).

We see that two topologically conjugate flows must be equivalent. In fact, Kuiper had solved the classification problem for topologically equivalence of flows $e^{tA}$ near $x_0=0$ in his paper \cite{kuiper1973flow} in 1973 (published in 1975). Before his paper, Kuiper and Robbin considered a generalized classification problem, that is, for general linear transformations, which can be regarded as discrete-time dynamical systems. To be precise, they wanted to classify linear transformations from $V$ to $\tti{V}$ up to the equivalence relation: $f\sim\ti{f}$ if $h\circ f=\ti{f}\circ h$. But the latter seems to be more difficult and remains unsolved (see Kuiper and Robbin \cite{kuiper1973endo}). 

Considering the real Jordan form of $A$, $\bb{R}^n$ decomposes into three invariant subspaces $V_+\oplus V_-\oplus V_0$ corresponding to the eigenvalues with positive, negative, zero real parts respectively. It's obvious that the point $x_0\neq 0$ lying in $V_+$ (resp. $V_-$) goes to infinity (resp. $0$) when acted by the flow $e^{tA}$ ($t\to +\infty$). These initial points lying in $V_+\oplus V_-$ tend to be easy to deal with, we call it ``hyperbolic case'' while the case of $V_0$ is called ``non-hyperbolic case''. The conclusion of the topological equivalence classification for linear flows given in Kuiper's paper \cite{kuiper1973flow} is (roughly stated):
\begin{tpeq}
Linear flows $e^{tA},e^{t\tti{A}}$ are topologically equivalent $\Longleftrightarrow$ $(\dim V_+,\dim V_-)$$=(\dim \tti{V}_+,\dim \tti{V}_-)$ and $A_0,\tti{A_0}$ are linear conjugate.
\end{tpeq}

By using this classification result, Ayala and Kawan give a complete topological conjugacy classification for real projective flows in \cite{ayala2014proj}, 2014 (the real projective flows are naturally induced by linear flows on Euclidean space via quotient map to $\bb{R}P^n$). As for the discrete-time case of projective flows, there is also only a partial result given by Kuiper in \cite{kuiper1976proj}, 1976. The classification for linear flows on vector bundles also has been studied. Its hyperbolic case is treated in Ayala, Colonius and Kliemann \cite{ayala2007vect}.

Although the topological equivalence classification for linear flows had been solved by Kuiper and independently Ladis, Ladis' proof \cite{ladis1973flow} was published in 1973 on a Russian journal which is not easily available today, while Kuiper's proof (1973) is more topological and a little bit subtle. Many relevant textbooks nowadays will contain the proof for hyperbolic case, but little material for non-hyperbolic case, while only their work has been cited. We've absorbed Kuiper's ideas, aiming at weaker classification: topological conjugacy classification, we can present a self-contained proof just using only linear algebra and elementary topology, even more concise and accessible to everyone (easy to find and understand).

The key to the proof in this paper is that, as for weaker classification, the period of each flow orbit becomes topologically invariant, which is not in Kuiper's proof. So Kuiper used some deeper topological notions to provide invariants (related to ratios of periods). However, once using period as invariant, we may get a concise proof. We also write this proof in detail. 

\begin{mt}\label{mainthm}
Given $n$-dimensional Euclidean spaces $V,\tti{V}$, and $A,\tti{A}$ real linear opertors on them respectively. Then,
linear flows $e^{tA},e^{t\tti{A}}$ are topologically conjugate $\Longleftrightarrow$ $(\dim V_+,\dim V_-)$$=(\dim \tti{V}_+,\dim \tti{V}_-)$ and $A_0,\tti{A_0}$ are linear conjugate,
where $V=V_+\oplus V_-\oplus V_0$ corresponding to the eigenvalues with positive, negative, zero real parts respectively, $A_+=A|_{V_+}$, $A_-=A|_{V_-}$, $A_0=A|_{V_0}$ and so on.
\end{mt}

The proof for hyperbolic case is canonical (see \cite{colo2014dyna}) which won't be contained in this paper but sketched now. Sufficiency: for flow $e^{tA_+}$, we construct a surface in $\bb{R}^n$ which transversally intersects each orbit, whence we topologically rectify the flows and show it's conjugate to $e^t$. This surface is defined by a well-chosen quadric form, named Lyapunov function.  Necessity: the points in $V_+$ will be sent to infinity when acted by flow $e^{tA}$ ($t\to +\infty$), so $h$ induces homeomorphism between $V_+,\tti{V}_+$, then a fact in topology (invariance of dimension, see prop.(\ref{invdim})) implies $\dim V_+=\dim \tti{V}_+$. However, the non-hyperbolic case, to show $A_0,\tti{A_0}$ are linear conjugate, is not easy. We have to find enough topological invariants to show the Jordan forms of them are exactly the same.

Following Kuiper, we first focus on the bounded orbits, they form an invariant space corresponding to those upper-left corners of Jordan blocks subordinate to zero real parts (section \ref{sect1}). Moreover, in order to show those corners (eigenvalues) are the same, we consider compact orbits. Just like rotation $e^{i\beta t}$, the eigenvalues related to a compact orbit have rational ratios, and periods help distinguish the irrational relation between eigenvalues (section \ref{sect2}). As for the eigenvalues of the same rational type, different from Kuiper's technique, we introduce a characteristic mapping $\chi$ to record period of each point. See the period of $(e^{2it} x_0,e^{3it} x_0)$ is $T=2\pi$ $(x_0\neq 0)$, the period of $(e^{2it} x_0, e^{4it} x_0)$ is $T=\frac{1}{2}\cdot2\pi$, we conclude that the image of $\chi$ is related to the greatest common divisors of eigenvalues. By using this observation, we find a good invariant $\dim \overline{\chi^{-1}(q)}$ to show those corners indeed coincide (section \ref{sect3}). What remains unknown are the sizes of Jordan blocks, thus we transcribe Kuiper's elegant proof in which he finds a topological notion recording their information. For example, on $xOy$-plane the orbits of the flow $\exp({t\begin{pmatrix}0&1\\0&0\end{pmatrix}})=\begin{pmatrix}1&t\\0&1\end{pmatrix}$ are zero, half $x$-axises and other horizontal lines. We see that although the points $(0,0),(1,0)$ are not on the same orbit, there are orbits (namely $y=c\to 0$) approach to them simultaneously. In general, those points correspond to the upper-left quarters of the Jordan blocks and therefore this approaching phenomenon records the final information we need (section \ref{sect4}). We'll apply topological view and coordinate view in turn, one for invariance and the latter for proof.

\section{Preliminaries}\label{preli}
\begin{defn}
Given an $n$-dimensional real (complex) linear space $V$, a {\bfseries flow} on $V$ is a continuous map $f:\bb{R}\times V\to V$ (denoted by $f(t,x)=f^tx$) satisfying $f^{s+t}=f^s\circ f^t, f^0=id$. An {\bfseries orbit} of $f^t$ is the subset $\ca{O}_f(x_0)=\{f^tx_0\}_{t\in\bb{R}}$ where $x_0$ is some point in $V$.
\end{defn}
\begin{defn}
Two flows $f^t,\ti{f}^t$ on $V,\tti{V}$ are called {\bfseries topologically equivalent}, if there exists a homeomorphism $h:V\to \tti{V}$, and for any $x_0$ there is a strictly increasing continuous map $\tau:\bb{R}\to\bb{R}$ such that $h\circ f^{\tau(t)}=\ti{f}^t\circ h\ (\forall t\in\bb{R})$ (sending orbit to orbit and preserving the orientation of time). They are called {\bfseries topologically conjugate}, if $h\circ f^t=\ti{f}^t\circ h\ (\forall t\in\bb{R})$.
\end{defn}

We introduce some facts in linear algebra and topology:
\begin{prop}[Invariance of dimension]\label{invdim}
Two Euclidean spaces $V,\tti{V}$ are homeomorphic if and only if $\dim{V}=\dim{\tti{V}}$.
\end{prop}
\begin{prop}\label{complex}
Two complex-linear-conjugate real matrices are real-linear-conjugate.
\end{prop}
\begin{prop}[Jordan form]
Given a linear operator $A$ of $n$-dimensional real(resp. complex) linear space $V$, there is a basis $\{v_i\}_{1\leq i\leq n}$ such that:
\begin{displaymath}
A(v_1,\dots,v_n)=(v_1,\dots,v_n)
\begin{pmatrix}
J_1& & & \\
 & J_2 & &\\
 & & \ddots &\\
 & & & J_k
\end{pmatrix}
\end{displaymath}
where the Jordan block $J$ is of the form:$J(\lambda)=
\begin{pmatrix}
\lambda&1& & \\
 & \lambda &1&\\
 & & \ddots &1\\
 & & & \lambda
\end{pmatrix}
$ if $\lambda\in\bb{R}$ (resp. if $\lambda\in\bb{C}$), else $J(\lambda)=
\begin{pmatrix}
\Lambda&E& & \\
 & \Lambda &E&\\
 & & \ddots &E\\
 & & & \Lambda
\end{pmatrix}$ where $\Lambda=\begin{pmatrix} Re(\lambda)&Im(\lambda) \\-Im(\lambda) & Re(\lambda)\end{pmatrix}$, $E=\begin{pmatrix}0&0\\1&0\end{pmatrix}$.
\end{prop}
\begin{rem}
The vectors $v_i$ corresponding to the blocks subordinate to eigenvalues with positive real parts span the invariant space $V_+$, and so on.
\end{rem}

We focus on non-hyperbolic case of classification:
\begin{mt}[Non-hyperbolic case]\label{mainthm2}
Conditions are the same as previous. Then,
linear flows $e^{tA},e^{t\tti{A}}$ are topologically conjugate $\Longrightarrow$ $A_0,\tti{A_0}$ are linear conjugate.
\end{mt}
\paragraph{Complexification}
Rather than $V,\tti{V}$ themselves, we consider their complexifications to extend scalar-multiplication to $\bb{C}$. That is, consider the functor $\bb{C}\otimes_{\bb{R}}\underline{\hbox to 3mm{}}$. Via the one-to-one mapping $v\mapsto 1\otimes v$, the invariant spaces, flows in real case can be identified with that in their complexifications. If we can verify that $A_0,\tti{A_0}$ (in fact $id\otimes A_0,id\otimes\tti{A_0}$) are complex linear-conjugate, then following prop.(\ref{complex}) they are real linear-conjugate in original spaces. Hence we only need to show the blocks (complex Jordan forms) subordinate to zero real parts of $A$ are the same as that of $\tti{A}$.

\section{Bounded Orbit Subspaces}\label{sect1}
Suppose there is a homeomorphism $h:V\to \tti{V}$ such that $h\circ e^{tA}=e^{t\tti{A}}\circ h$. By translation, we assume $h(0)=0$. Roughly speaking, a {\bfseries topological notion} (or {\bfseries topologically invariant}) is a notion can be defined topologically in $V$ and parallelly in $\tti{V}$, whence $h$ may ``map'' this notion to its parallel.

\begin{defn}
A subset $W$ of $V$ is called an {\bfseries orbit family}, if for any $x\in W$, we have $Ax\in W$ and $\ca{O}(x)\subseteq W$. An orbit family $W$ is called an {\bfseries orbit subspace}, if $W,h(W)$ are both linear subspaces.
\end{defn}

We introduce some topological notions at first: $\bb{B}=\{x\in V\ |\ \ca{O}(x)\textrm{ is bounded}\}$, $\bb{D}=\{x\in V\ |\ \ca{O}(x)\textrm{ is a single point}\}$ (It's obvious that $h$ induces homeomorphisms from $\bb{B}$ to $\bb{\tti{B}}$ and from $\bb{D}$ to $\bb{\tti{D}}$). One can check they are linear subspaces, whence they are orbit subspaces.

We now consider orbit subspaces $B$ in $\bb{B}$ (assign $h(B)$ to be its parallel notion) and $D=\bb{D}\cap B$. We assert that $B$ corresponds those upper-left corners of some blocks subordinate to zero real parts. To show that explicitly, we tend to find the {\bfseries coordinate representation} of $B$.

Suppose $A|_B$ has Jordan form $
\begin{pmatrix}
J_1& &\\
 &\ddots &\\
 & & J_v
\end{pmatrix}$, $z\in B$ having nonzero coordinate $(z_1,z_2,\dots, z_m)^T$ corresponding to one Jordan block $
\begin{pmatrix}
\lambda&1& \\
 & \ddots &1\\
 & & \lambda
\end{pmatrix}_{m\times m}$.
According to the definition of $B$, $\ca{O}(z)=\{e^{tA}z\}_{t\in\bb{R}}$ is bounded. That is,
\begin{displaymath}
e^{t\lambda}
\begin{pmatrix}
1&\frac{t}{1!}&\ldots&\frac{t^{m-1}}{(m-1)!}\\
 & 1 &\ddots&\vdots\\
 & & \ddots &\frac{t}{1!}\\
 & & & 1
\end{pmatrix}
\begin{pmatrix}
z_1\\z_2\\ \vdots\\ z_m
\end{pmatrix}
\end{displaymath}
is bounded. So we have $Re(\lambda)=0$ and $z_2=\cdots=z_m=0$. So $B\subseteq V_0$, and $A|_B$ can be diagonalized. Each orbit $\ca{O}(z)$ in $B$ has coordinates
\begin{displaymath}
\{(y,e^{i\beta_1t}z_1,\dots,e^{i\beta_vt}z_v)\ |\ \ t\in\bb{R}\}\textrm{ where }\beta_i\in (0,2\pi),y\in D,z_i\in\bb{C}
\end{displaymath}
The following two sections focus on showing the eigenvalues are the same, i.e. $(\beta_1,\dots,\beta_v)=(\ti{\beta_1},\dots,\ti{\beta_{v'}})$.

\section{Rational Equivalence Classes}\label{sect2}
Two real numbers $\beta_i,\beta_j$ are {\bfseries rational equivalent}, if $\beta_i\bb{Q}=\beta_j\bb{Q}$ (as cosets in multiplication group $\bb{R}/\bb{Q}$). Then $\{\beta_1,\dots,\beta_v\}$ are divided into several equivalence classes. For $i$-th class $\{\beta_{k_1},\dots,\beta_{k_m}\}$, by coordinate form derived from the last section, we define
\begin{displaymath}
C_i=\bigcup_{z_{k_j}\in\bb{C}}\{(0,\dots,e^{i\beta_{k_1}t}z_{k_1},\dots,e^{i\beta_{k_m}t}z_{k_m},\dots)\ |\ \ t\in\bb{R}\}
\end{displaymath}
(All are zero except the positions corresponding to $z_{k_1},\dots,z_{k_m}$)

Obviously, $C_i$ is an orbit subspace of $B$. And the minimal positive period of each orbit, issuing from a nonzero point in $C_i$, is an integral multiple of $\frac{2\pi}{\beta_{k_j}}$ (if $z_{k_j}\neq0$).

We have decomposition $B=D\oplus C_1\oplus C_2\cdots \oplus C_r$. If one consider the coordinate form of each compact orbit in $B$, we then obtain that $D\oplus C_1\cup D\oplus C_2 \cup\cdots\cup D\oplus C_r$ is the union of all the compact orbits in $B$ (since the related eigenvalues of an orbit with finite period must have ration ratios to each other).

Suppose $C_1,C_2,\dots,C_r$ correspond to the cosets $\eta_1\bb{Q},\eta_2\bb{Q},\dots,\eta_r\bb{Q}$. We assert those cosets are topological notions.

From the view of topology, we assign each non-degenerate compact orbit in $B$ with $\frac{2\pi}{T}\bb{Q}$, where $T$ is its minimal positive period. According to the knowledge about the phase curves of autonomous equations, we know ``minimal positive period'' is a topological notion (since if $T$ is a period of point $x$, then $\ti{f}^T\circ h(x)=h\circ f^T(x)=h(x)$ implies $T$ is also a period of $h(x)$).

Under this assignment, $C_1,C_2,\dots,C_r$ still correspond to $\eta_1\bb{Q},\eta_2\bb{Q},\dots,\eta_r\bb{Q}$. So they are topologically invariant. Moreover, since $h$ maps compact orbit to compact orbit, singleton orbit to singleton orbit, those invariants shows that $h$ induces homeomorphism
\begin{displaymath}
h:\ D\oplus C_i \longrightarrow \tti{D}\oplus \tti{C_i}
\end{displaymath}

\section{Elements in Rational Equivalence Classes}\label{sect3}
Denote $C=C_i$, for $D\oplus C$, let subset $\Sigma_0$ be the union of non-degenerate orbits (topological notion), i.e we have coordinate form $\Sigma_0=D\oplus C\setminus D\oplus 0$.

Consider {\bfseries characteristic mapping}
\begin{displaymath}
\chi:\ \Sigma_0 \longrightarrow \bb{R} \ \textrm{ sending }\  x\mapsto \frac{2\pi}{T}
\end{displaymath}
where $T$ is the minimal positive period of $\ca{O}(x)$. So this mapping is also a topological notion. We call each image of $\chi$ by {\bfseries singular value}.

From the view of coordinate, by omitting the coordinates not belonging to $D\oplus C$, we may write down this concise form $C=\cup_{z_i\in\bb{C}}\{(0,e^{i\beta_1t}z_1,\dots,e^{i\beta_mt}z_m)\ |\ \ t\in\bb{R}\}$ where $\beta_i\in (0,2\pi)$ are rational equivalent.
Suppose $(\beta_1,\dots,\beta_m)=\beta\cdot(p_1,\dots,p_m)$ where $p_i$ are positive integers and $\gcd(p_1,\dots,p_m)=1$.

We now are able to write down the explicit form of the characteristic mapping.
For $x=(y,e^{i\beta_1t}z_1,\dots,e^{i\beta_mt}z_m)\in D\oplus C$, if $z_{i_1},\dots,z_{i_k}$ are all the nonzero elements in $z_1,\dots,z_m$, then 
\begin{displaymath}
\chi(x)=\beta\cdot\gcd(p_{i_1},\dots,p_{i_k})
\end{displaymath}

\begin{prop}\label{charinv}
For any singular value $q$, the closure of preimage (topological notion) $\overline{\chi^{-1}(q)}$ is a linear space. Moreover, $\dim_{\bb{C}}\overline{\chi^{-1}(q)}-\dim_{\bb{C}}D$ is the cardinality of $\{i\ |\ 1\leq i\leq m,\ \frac{q}{\beta}\ |\ p_i\}$.
\end{prop}
\begin{proof}
From the view of coordinate, we have $\overline{\chi^{-1}(q)}=\{(y,\dots,z_{i_1},\dots,z_{i_k},\dots)\ |\ k>0,1\leq i_1<\cdots<i_k\leq m,\gcd(p_{i_1},\dots,p_{i_k})=q/\beta, z_{i_1},\dots,z_{i_k}\in\bb{C},y\in D\}$. Consider the index set $\{j_1,\dots,j_l\}=\{j\ |\ p_j\textrm{ is divisible by }q/\beta\}$. Since $q$ is a singular value, $q/\beta$ is the greatest common divisor of some $p_i$. Hence we have $\gcd(p_{j_1},\dots,p_{j_l})=q/\beta$. Therefore $\overline{\chi^{-1}(q)}=D\oplus \bb{C}z_{j_1}\oplus\cdots \oplus\bb{C}z_{j_l}$ is a linear space and $\dim\overline{\chi^{-1}(q)}-\dim D$ is the number of the $p_i$ divisible by $q/\beta$.
\end{proof}

\begin{lem}\label{eigensame}
The eigenvalues in $B$ and in $\tti{B}$ are exactly the same.
\end{lem}
\begin{proof}

The number of zero eigenvalues is $\dim D=\dim\tti{D}$ (by the invariance of dimension, prop.(\ref{invdim})). Hence it suffices to show the eigenvalues in $D\oplus C$ and in $\tti{D}\oplus \tti{C}$ are exactly the same. The numbers of nonzero eigenvalues are also the same. Suppose $(\beta_1,\dots,\beta_m)=\beta\cdot(p_1,\dots,p_m)$, $\gcd(p_1,\dots,p_m)=1$, and $(\ti{\beta_1},\dots,\ti{\beta_m})=\ti{\beta}\cdot(\ti{p_1},\dots,\ti{p_m})$, $\gcd(\ti{p_1},\dots,\ti{p_m})=1$.

According to the invariance of period, $\chi(x)=\chi(h(x))$. The singular value sets of $\Sigma_0$ and $\tti{\Sigma_0}$ coincide. So we know $\beta=$the minimal singular value$=\ti{\beta}$. Assuming $p_1\leq\dots\leq p_m$, $\ti{p_1}\leq\dots\leq\ti{p_m}$, we have $\beta\cdot p_m=$the maximal singular value$=\ti{\beta}\cdot\ti{p_m}$, whence $p_m=\ti{p_m}$.

Assume we already have $(p_{k+1},\dots,p_m)=(\ti{p_{k+1}},\dots,\ti{p_m})$. Since $\beta\cdot p_k$ is singular value, by prop.(\ref{charinv}) we have 
\begin{displaymath}
\sharp\{i\ |\ 1\leq i\leq m,p_k\ |\ p_i\}=\dim\overline{\chi^{-1}(\beta p_k)}-\dim D=\sharp\{i\ |\ 1\leq i\leq m,p_k\ |\ \ti{p_i}\}
\end{displaymath}
We must have $p_k\leq \ti{p_k}$, or else $\sharp\{i\ |\ 1\leq i\leq m,p_k\ |\ p_i\}=\sharp\{i\ |\ k< i\leq m,p_k\ |\ \ti{p_i}\}=\sharp\{i\ |\ k<i\leq m,p_k\ |\ p_i\}$ (the second equality follows induction) which is a contradiction.

By the same argument, we have $\ti{p_k}\leq p_k$. So $p_k=\ti{p_k}$. In conclusion of the induction, we have $(\beta_1,\dots,\beta_m)=(\ti{\beta_1},\dots,\ti{\beta_m})$.
\end{proof}

\section{The Sizes of Jordan Blocks}\label{sect4}
For an orbit family $W$ of $V$, define relation $\ca{R}_W$: for $x,y\in W$, $x\ca{R}_W y$, if for any neighborhoods $U_x,U_y$ of $x,y$ in $W$, there is an orbit in $W$ intersects both $U_x,U_y$. Define operators $X,Y$:
\begin{align*}
X(W)&=0\cup\{x\in W\ |\ \exists y\in W\setminus \ca{O}(x), x\ca{R}_W y\}\\
Y(W)&=0\cup\{x\in W\ | x\ca{R}_W 0\}\subseteq X(W)
\end{align*}
It can be checked that $X(W),Y(W)$ are also orbit families, and if $W$ is an orbit subspace, so is $Y(W)$.

Define operator $Z$:
\begin{displaymath}
Z^{(1)}(W)=Y(W),\ Z^{(2r)}=Z^{(r)}(X(W)),\ Z^{(2r+1)}=Z^{(r)}(Y(W)),\ r\geq 1
\end{displaymath}
We observe that once regarding $X,Y$ as $0,1$ in binary system and writing $m\in\bb{N}$ in binary string by $X,Y$, then $Z^{(m)}$ is just the composition of those binary digits in turn. For example, $m=6$, then $m=YYX$ in binary format, and $Z^{(m)}(W)=Y(Y(X(W)))$.

Consider the invariant subspace $W_m$ corresponding to the Jordan block $J=
\begin{pmatrix}
i\beta&1& \\
 & \ddots &1\\
 & & i\beta
\end{pmatrix}_{m\times m}$. It's an orbit subspace.
\begin{prop}\label{redumat}
If the coordinate form (under the Jordan basis) of $W_m$ is $\{(x_1,\dots,x_m)\ |\ x_i\in\bb{C}\}$, then 
\begin{align*}
X(W_m)&=\{(x_1,\dots,x_{\lceil\frac{1}{2}m\rceil},0,\dots,0)\ |\ x_i\in\bb{C}\}=W_{\lceil\frac{1}{2}m\rceil}\\
Y(W_m)&=\{(x_1,\dots,x_{\lfloor\frac{1}{2}m\rfloor},0,\dots,0)\ |\ x_i\in\bb{C}\}=W_{\lfloor\frac{1}{2}m\rfloor}
\end{align*}
That is, if we have binary strings $\alpha,\beta$, then
\begin{align*}
\alpha X(W_{\beta X})=\alpha(W_\beta)&, \alpha X(W_{\beta Y})=\alpha(W_{\beta+1})\\
\alpha Y(W_{\beta X})=\alpha(W_\beta)&, \alpha Y(W_{\beta Y})=\alpha(W_\beta)
\end{align*}
\end{prop}
\begin{proof}
We first consider the case $m=2r+1$ and verify the proposition following:
\begin{prop}\label{propeq1}
$\begin{pmatrix}
\frac{1}{r!}&\frac{1}{(r+1)!}&\ldots&\frac{1}{(2r)!}\\
\vdots&\vdots&\ddots&\vdots\\
\frac{1}{1!}&\frac{1}{2!}&\ldots&\frac{1}{(r+1)!}\\
\frac{1}{0!}&\frac{1}{1!}&\ldots&\frac{1}{r!}
\end{pmatrix}
\begin{pmatrix}
x_0\\x_1\\\vdots\\x_r
\end{pmatrix}
=
\begin{pmatrix}
y_0\\y_1\\\vdots\\y_r
\end{pmatrix}
\Rightarrow
x_0=(-1)^r y_r+\sum_{i<r} a_i y_i
$ where $a_i$ are constants depending on $r$.
\end{prop}
\begin{proof}
Let $P(t)=\frac{x_0}{r!}t^r+\cdots+\frac{x_r}{(2r)!}t^{2r}$. Then this system of linear equations is $(P(1),P'(1),\dots,P^{(r)}(1))=(y_0,y_1,\dots,y_r)=y$. Once $y=0$, we have $(t-1)^{r+1}|P$, while $t^r|P$ by definition. Hence $P=0$, which implies the matrix is invertible. Moreover, set $y_0=y_1=\cdots=y_{r-1}=0$, we have $(t-1)^rt^r|P$ whence $P(t)=(-1)^r\frac{x_0}{r!}(t-1)^rt^r$, and $x_0=(-1)^r y_r$ follows. So by Cramer's rule, in general $x_0=(-1)^r y_r+\sum_{i<r} a_i y_i$.
\end{proof}
The explicit form of the system $e^{tJ}x=y(t\neq0)$ is:
\begin{equation}\label{sys1}
\left\{\begin{array}{rrcrl}
x_1+\cdots+\frac{t^{r-1}}{(r-1)!}x_r+&\frac{t^r}{r!}x_{r+1}&+\cdots+&\frac{t^{2r}}{(2r)!}x_{2r+1}&=y_1e^{-i\beta t}\\
\ldots\quad\ldots&\quad&\ldots&\ldots&\ldots\\
x_r+&\frac{t}{1!}x_{r+1}&+\cdots+&\frac{t^{r+1}}{(r+1)!}x_{2r+1}&=y_re^{-i\beta t}\\
&x_{r+1}&+\cdots+&\frac{t^r}{r!}x_{2r+1}&=y_{r+1}e^{-i\beta t}\\
&&\ldots&\ldots&\ldots\\
&&&x_{2r+1}&=y_{2r+1}e^{-i\beta t}
\end{array}\right.
\end{equation}
Multiply the $i$-th equation by $t^{i-r-1}$ ($i=1,\dots,r+1$), then
\begin{equation}\label{sys2}
\left\{\begin{array}{rrcrl}
\frac{1}{r!}x_{r+1}+&\frac{1}{(r+1)!}tx_{r+2}&+\cdots+&\frac{1}{(2r)!}t^rx_{2r+1}&=\varepsilon_1\\
\frac{1}{(r-1)!}x_{r+1}+&\frac{1}{r!}tx_{r+2}&+\cdots+&\frac{1}{(2r-1)!}t^rx_{2r+1}&=\varepsilon_2\\
\ldots&\ldots&\ldots&\ldots&\ldots\\
\frac{1}{1!}x_{r+1}+&\frac{1}{2!}tx_{r+2}&+\cdots+&\frac{1}{(r+1)!}t^rx_{2r+1}&=\varepsilon_r\\
\frac{1}{0!}x_{r+1}+&\frac{1}{1!}tx_{r+2}&+\cdots+&\frac{1}{r!}t^rx_{2r+1}&=y_{r+1}e^{-i\beta t}\\
\end{array}\right.
\end{equation}

Part(I).
For any $x^0\in X(W_m)$, suppose $x^0\ca{R}_W y^0, y^0\notin\ca{O}(x^0)$. Then we have sequences of points: $x^n\to x^0,y^n\to y^0$ where $y^n=e^{t_nJ}x^n, n\to+\infty$. Since $y^0\notin \ca{O}(x^0)$, we may assume $t_n\to \infty$. By the definition of system (\ref{sys2}), we have $\varepsilon_1^n,\dots,\varepsilon_r^n\to 0$. And taking the Cramer's rule on system (\ref{sys2}), we have $x_{r+2}^n,\dots,x_{2r+1}^n\to0$, which implies $x_{r+2}^0=\cdots=x_{2r+1}^0=0$.

Part(II).
Conversely, given $x=(x_1,\dots,x_{r+1},0,\dots,0)$. Set $t_n=n$, and set
\begin{displaymath}
\left\{\begin{array}{rl}
y_1&\neq e^{i\beta t}(x_1+\frac{t}{1!}x_2+\cdots+\frac{t^r}{r!}x_{r+1})\\
y_2&\neq e^{i\beta t}(x_2+\cdots+\frac{t^{r-1}}{(r-1)!}x_{r+1})\\
&\ldots\quad\ldots\\
y_r&\neq e^{i\beta t}(x_r+tx_{r+1})\\
y_{r+1}&= (-1)^rx_{r+1}
\end{array}\right.
\end{displaymath}($\forall t\in\bb{R}$).
Set $y=(y_1,\dots,y_{r+1},0,\dots,0)$. By definition, $y\notin\ca{O}(x)$.

According to prop.(\ref{propeq1}), given $x_1,\dots,x_r,y_1,\dots,y_{r+1}$, the system (\ref{sys1}) has unique solution for each $n>0$:
\begin{align*}
x^n&=(x_1,\dots,x_r,x_{r+1}^n,\dots,\dots,x_{2r+1}^n)\\
y^n&=(y_1,\dots,y_r,y_{r+1},y_{r+2}^n,\dots,y_{2r+1}^n)
\end{align*}
By the same argument in Part(I), we have $x_{r+2}^n,\dots,x_{2r+1}^n\to0,y_{r+2}^n,\dots,y_{2r+1}^n\to0$. According to prop.(\ref{propeq1}), $x_{r+1}^n=(-1)^r y_{r+1}+\sum_{i<r} a_i \varepsilon_i\to(-1)^r y_{r+1}=x_{r+1}$ ($n\to +\infty$). Hence we have $x^n\to x, y^n\to y, y^n\in\ca{O}(x^n)$, that is, $x\ca{R}_W y$. In conclusion, $X(W_m)=W_{\lceil\frac{1}{2}m\rceil}$.

By similar argument, we can also deduce the case $m=2r$ and the part for $Y$.
\end{proof}
By considering coordinate forms, the prop.(\ref{redumat}) deduces this corollary immediately.
\begin{cor}\label{cor1}
1. $Z^{(k)}(W_m)\left\{\begin{array}{ll}=0&,k\geq m\\\neq0&,k<m\end{array}\right. 2. $ $Z^{(k)}(V)\cap V_0=Z^{(k)}(V_0)$ 3. $Z^{(k)}(W_m)\cap \bb{B}\left\{\begin{array}{ll}=0&,k\geq m\\\cong\bb{C}&,k<m\end{array}\right.$
\end{cor}(recall that $\bb{B}$ is the subspace corresponding to those upper-left corners of all the blocks subordinate to zero real parts, see section \ref{sect1}.)\qed

\begin{mt}[Non-hyperbolic case]
Stated in the main theorem, section \ref{preli}.
\end{mt}
\begin{proof}
Consider the topologically invariant orbit subspaces $F^{(k)}=Z^{(k)}(V)\cap\bb{B}$, that is, $h$ induces homeomorphism $F^{(k)}\to\tti{F^{(k)}}$.

The cor.(\ref{cor1}) tells us $F^{(k)}$ is the direct sum of the one-dimensional eigenspaces corresponding to upper-left corners of all the Jordan blocks belonging to $V_0$ whose order$>k$. According to lemma (\ref{eigensame}), the eigenvalues of $F^{(k)},\tti{F^{(k)}}$ are exactly the same. Let $N(\lambda,m)$ denote the number of Jordan blocks of order $m$ subordinate to eigenvalue $\lambda$ ($Re(\lambda)=0$). Then the equality of multiplicities for $\lambda$ in $F^{(k)},\tti{F^{(k)}}$ is:
\begin{displaymath}
\sum_{m>k}N(\lambda,m)=\sum_{m>k}\tti{N}(\lambda,m)
\end{displaymath}
Hence we have $N(\lambda,m)=\tti{N}(\lambda,m)$, that is, $A_0,\tti{A_0}$ are complex-linear-conjugate. By the prop.(\ref{complex}), we know the original real linear operators $A_0,\tti{A_0}$ are real-linear-conjugate.
\end{proof}


%

\end{document}